\documentclass[11pt]{amsart}
\usepackage{pstricks,pstricks-add,pst-node} 
\begin{document}
\date{\today}
\title[Abelian permutation groups with graphical representations]{Abelian permutation groups with graphical representations}
\author{Mariusz Grech, Andrzej Kisielewicz}
\address{Institute of Mathematics, University of Wroclaw \\ 
pl.Grunwaldzki 2, 50-384 Wroclaw, Poland}
\email{[Mariusz.Grech,Andrzej.Kisielewicz]@math.uni.wroc.pl}
\thanks{Supported in part by Polish NCN grant 2016/21/B/ST1/03079}

\newtheorem{Theorem}{Theorem}[section]
\newtheorem{Lemma}[Theorem]{Lemma}
\newtheorem{Corollary}[Theorem]{Corollary}

\newcommand{\Aut} {\mbox{\rm Aut}}
\newcommand{\Orb} {\mbox{\rm Orb}}
\newcommand{\Cay} {\mbox{\rm Cay}}
\newcommand{\clo}[1]{\overline{#1}}
\newcommand{\<} {\langle}
\renewcommand{\>} {\rangle}


\begin{abstract}
In this paper we characterize permutation groups that are automorphism groups of coloured graphs and digraphs and are abelian as abstract groups. This is done in terms of basic permutation group properties. Using Schur's classical terminology, what we provide is characterizations of the classes of $2$-closed and $2^*$-closed abelian permutation groups. This is the first characterization concerning these classes since they were defined.
\end{abstract}


\maketitle


\section{Introduction}



This paper is motivated by the K\"onig's problem for permutation groups. Recall that K\"onig's problem for groups asked which finite groups are the automorphism groups of (simple) graphs. This question, in its abstract version, was quickly answered by Frucht 
who showed that every abstract group is isomorphic to the automorphism group of some graph. The concrete version, of more combinatorial flavour, asks which finite permutation groups happen to be the automorphism groups of graphs. This version turns out to be much harder (see  \cite{bab2,god1}). In the abstract version we look for a graph with an arbitrary number of vertices whose automorphism group is \emph{isomorphic} to a given (abstract) group. In the concrete version we are given a permutation group $G=(G,X)$ acting on a set $X$ of elements and we are looking for a graph $(\Gamma,X)$ with the same set $X$ of vertices whose automorphisms \emph{are precisely} the permutations in $G$. 

It is important to understand fully the aims and motivation of this paper. There are related active areas of research concerning graphical representations of abstract groups, Cayley graphs and automorphism groups of circulant graphs  (see \cite{bab1,bab2,god,mor} and  \cite{Aru,DTW,DSV,HKMM,spi} to mention just a few most recent).
All they are closely connected with the problem considered in this paper. Yet, we approach the topic from a different point of view. Our main interest is in \emph{permutation groups} and classifying them by their natural representations as the automorphism groups of coloured graphs. This refers directly to the direction of research suggested in Wielandt's \cite{wie}. Many natural applications of permutation groups, especially in fields outside mathematics, concern the ways they act rather than their group structure (cf. \cite{KP,KPR,kis2}). While there is a lot of knowledge on transitive permutation groups 
(mainly due to research connected with the classification of finite simple groups), very little is known about intransitive groups and  their action on different orbits. Note that most graphs are not vertex transitive and the action of the automorphism group on different orbits may be very different and essential with regard to various graph properties.  

Our general aim is to contribute to the study of permutation groups on their own with various possible applications in mind. 





It is easy to see that some permutation groups, like the alternating groups $A_n$ on $n$ elements or the groups $C_n$ generated by the cyclic permutation $(1,2,\ldots,n)$, are not the groups of automorphisms of any graph (on $n$ elements). For a long time there was no progress in the concrete version of K\"onig's problem. 
Only research on the so called \emph{Graphical Regular Representation}  of groups (GRR) turned out a more systematic approach to the problem. The final result of this extensive study by Godsil \cite{god}, even if it concerned representations of abstract groups, may be interpreted as the description of those \emph{regular permutation groups} that are the automorphism groups of graphs. Analogous result for automorphism groups of directed graphs has been obtained in Babai~\cite{bab1}.  

The next natural class of permutation groups to study from this point of view is the class of cyclic permutation groups, that is, those generated by a single permutation. The K\"onig's problem for this class turned out not so easy as it could seem at the first sight. After some partial results containing errors and wrong proofs (\cite{MSS1,MSS2} corrected in \cite{gre}), the final result has been obtained only recently 
(announced in \cite{GK4}). The full description turns out to consist of seven technical conditions concerning possible lengths of the orbits. 

To obtain the result, we have applied the mentioned Wielandt's approach to start from considering the invariance groups of families of binary relations rather than the automorphism groups of simple graphs. In the language of graphs these are the automorphism groups of (edge) coloured graphs. This approach is more natural. For example, we have obtained an easy-to-formulate result that a cyclic permutation group $G$ is the automorphism group of a coloured graph if and only if for every nontrivial orbit $O$ of $G$ there exists another orbit $Q$ such that $\gcd(|O|, |Q|) \geq 3$. The proof yields also the result that if a cyclic permutation group is the automorphism group of a coloured graph, than it is the automorphism group of a coloured graph that uses at most $3$ colours. Only then, one may consider the question for which cyclic permutation groups $2$ colours suffice, which turns out a rather technical one. 

The next natural class to attack in the concrete version of K\"onig's problem is that of abelian permutation groups. 
The survey \cite{mor} reports a result by Zelikovskij \cite{zel} where the solution of K\"onig's problem for a large class of abelian permutation groups (namely, those whose order is not divisible by $2,3$ or $5$) is provided. There is no English translation of \cite{zel}, so the survey quotes only the English summary. The restriction means that the lengths of the orbits must be relatively prime to $30$, and its apparent aim is to avoid technical complications. The worse thing is that the result as stated is false. In \cite{GK5} we demonstrate a counterexample and point out the false algebraic assumption used by Zelikovskij in the proof. 

In this paper we make the first step to find a correct characterization of the abelian permutation groups that are the automorphism groups of graphs. 
As before, we start from characterizing those abelian permutation groups that are the automorphism groups of coloured graphs and digraphs. 
Again, it turns out that this can be done in a quite nice way in terms of basic properties of permutation groups.

Our characterizations use a technical notion specific for intransitive permutation groups.
For a given permutation group $G$ we say that a permutation $\sigma$ is \emph{$2$-orbit compatible} with $G$ if for every pair of orbits $O$ and $Q$ of $G$ there is a permutation $\sigma'\in G$ such that $\sigma$ and $\sigma'$ act in the same way on $O\cup Q$. The group $G$ is $2$-\emph{orbit-closed} if every permutation that is $2$-{orbit-compatible} with $G$ belongs to $G$. Every transitive permutation group is trivially $2$-orbit-closed, but permutation groups containing more than $2$ orbits may not be.

We prove that an abelian permutation group $A$ is the automorphism group of an edge-coloured directed graph 
if and only if $A$ is $2$-orbit-closed (Theorem~\ref{th:main2}). 
Furthermore, we prove that an abelian permutation group $A$ is the automorphism group of an edge-coloured (simple) graph
if and only if $A$ is $2$-orbit-closed and satisfies an additional condition concerning groups induced by $A$ on its orbits   (Theorem~\ref{th:main}). Our main tool is the subdirect sum decomposition of intransitive groups, which is recalled for convenience of the reader in Section~\ref{subd}. 

The results of this paper have been announced in \cite{GK5}.

\section{Terminology}

For standard notions and terminology of permutation groups see e.g. \cite{DM}. We use the notation $G=(G,X)$ to denote a permutation group $G$ acting on a finite set $X$. Permutation groups are considered up to permutation isomorphism, i.e., two groups that are permutation isomorphic (in the sense of \cite[p. 17]{DM}) are treated as identical. In particular we usually assume that $X=\{1,2,\ldots,n\}$, and by $S_n$ and $A_n$ we denote 
the full \emph{symmetric group} and the \emph{alternating group}, respectively, acting on $X$. By $C_n$ we denote the subgroup of $S_n$ generated by the cyclic permutation $(1,2,\ldots,n)$. By $I_n$ we denote the \emph{trivial} permutation group acting on $n$ elements, that is, the subgroup of $S_n$ containing the identity permutation only.  

A $k$-\emph{coloured digraph} $\Gamma = (X, \gamma)$ is a set $X$ (of vertices) with a function $\gamma : X\times X \to \{0,1,\ldots, k-1\}$. 
If $\gamma$ is a function from the unordered pairs $P_2(X)$ of the set $X$ to $\{0,1,\ldots, k-1\}$, 
then $\Gamma$ is called a $k$-\emph{coloured graph}. They may be viewed as the complete digraph or the complete graph, respectively, on a set $X$, whose edges are coloured with at most $k$ different colours. In the case $k=2$, they may be identified with simple digraphs and graphs.

The $i$-th \emph{degree} of a vertex $x\in X$, denoted $d_i(x)$, is the number of edges in colour $i$ incident with $x$. The $k$-tuple $(d_0(x),\ldots,d_{k-1}(x))$ is referred to as the \emph{$k$-tuple of colour degrees} of $x$. In coloured digraphs we may distinguish also outdegrees and indegrees, and the corresponding $k$-tuples. 

A permutation $\sigma$ of $X$ is an \emph{automorphism} of $\Gamma$, if it preserves the colours of edges in $\Gamma$. Obviously, each automorphism preserves also 
the $k$-tuples of colour degrees of vertices. The automorphisms of $\Gamma$ form a permutation group, which is denoted  $\Aut(\Gamma)$. We say also that $\Gamma$ \emph{represents} (graphically) the permutation group $\Aut(\Gamma)$. We note that not every permutation group is representable by a coloured graph or digraph. For example, the alternating group $A_n$ is not representable for any $n>3$ (neither by a coloured graph nor by a digraph). The cyclic group $C_n$ is not representable by a coloured graph (for any $n>2$), but it is representable by a ($2$-coloured) digraph. 

Given a permutation group $G=(G,X)$, by $\Orb(G,X) = \Orb(G)$ we denote the coloured digraph in which two edges have the same colour if and only if they belong to the same orbit of $G$ in its action on $X\times X$ (i.e. \emph{orbital}). Similarly, by $\Orb^*(G,X) = \Orb^*(G)$ we denote the coloured graph in which two edges have the same colour if and only if they belong to the same orbit of $G$ in its action on $P_2(X)$.

It is easy to see that 
$\Aut(\Orb^*(G)) \supseteq \Aut(\Orb(G))\supseteq G$ as groups of permutations over $X$. The first group is called the $2^*$-\emph{closure} of $G$, while the second group is called the $2$-\emph{closure} of $G$. When a permutation group happens to be equal to its $2^*$-closure or $2$-closure, then it is called $2^*$-closed or $2$-closed, respectively. These groups are the largest permutation groups with the given set of orbits on $P_2(X)$ or on $X\times X$, respectively (see \cite{bab2,wie}; according to Wielandt \cite{wie} the notion of 2-closure as a tool in the study of permutation groups was introduced by I. Schur).

If $G$ is not $2^*$-closed, then $G$ is not the automorphism group of any (coloured) graph. 
Otherwise, there may be various coloured graphs $\Gamma$ such that $G = \Aut(\Gamma)$. Yet, each such graph can be obtained from $\Orb^*(G)$ by identifying some colours. In particular, if $G = \Aut(\Gamma)$ for a simple graph $\Gamma$, then $\Gamma$ can be obtained from $\Orb^*(G)$ by identifying some colours with $1$ (corresponding to edges), and other colours with $0$ (non-edges).

We define $GR(k)$ to be the class of all permutation groups that are automorphism groups of coloured graphs using at most $k$ colours. The union $GR = \bigcup_{k\geq 1}GR(k)$ is just the class of $2^*$-closed permutation groups. Similarly, we define $DGR(k)$ as the class of all permutation groups that are automorphism groups of coloured digraphs using at most $k$ colours. The union $DGR = \bigcup_{k\geq 1}DGR(k)$ is just the class of $2$-closed permutation groups.


Similar terminology may be introduced starting from orbits of a permutation group $(G,X)$ on $m$-tuples rather than on ordered pairs. Then in place of the orbitals we consider $m$-orbits (meaning the orbits in the action of $G$ on the Cartesian power $X^m$). Analogously, we can define $m$-closure of permutation groups. Then, $m$-closed groups and $m^*$-closed groups can be viewed as the invariance groups of families of relations. This is the point of view represented in Wielandt's \cite{wie}. (Note that the terminology here is not well established, and for example, in \cite{SV} the definitions of the same notions are different).  

While this seems pretty natural topic in the area of permutation groups not much has been done so far in it. The reason is that, on the one hand, the topic turned out to be rather hard, and on the other hand, the main stream of research in permutation groups was focused so far on delivering tools for the classification of finite simple groups, and this restricted research to transitive groups. Meanwhile, this should be again emphasized, the concrete version of K\"onig's problem and Wielandt's approach \cite{wie} are strictly connected with the fact that the solution may depends on the number of orbits and the properties of action on them.

\subsection{Subdirect sum decomposition}\label{subd}

A natural tool in the study of intransitive permutation groups is the subdirect sum of permutation groups. Given two permutation groups $G\leq S_n$ and $H\leq S_m$, the \emph{direct sum} $G\oplus H$ is the permutation group on $\{1,2,\ldots, n+m\}$ defined as the set of permutations $\pi =(\sigma,\tau)$ such that
$$
\pi(i) =
\left\{\begin{array}{ll}
\sigma(i), & \mbox{if }  i\leq n\\
n+\tau(i-n), & \mbox{otherwise.}
\end{array}\right.
$$
Thus, in $G\oplus H$, permutations of $G$ and $H$ act independently in a natural way on a disjoint union of the base sets of the summands. 

We introduce the notion of the \emph{subdirect sum} following \cite{GK} (and the notion of \emph{intransitive product} in \cite{kis}). 
Let $H_1\lhd\; G_1\leq S_n$ and $H_2\lhd\; G_2 \leq S_m$ be permutation groups such that $H_1$ and $H_2$ are normal subgroups of $G_1$ and $G_2$, respectively. Suppose, in addition, that factor groups $G_1/H_1$ and $G_2/H_2$ are (abstractly) isomorphic and $\phi :  G_1/H_1 \to G_2/H_2$ is the isomorphism mapping. Then, by
$$
G = G_1[H_1] \oplus_\phi G_2[H_2]
$$
we denote the subgroup of $G_1 \oplus G_2$ consisting of all permutations $(\sigma,\tau)$, $\sigma\in G_1, \tau\in G_2$ such that $\phi(\sigma H_1) = \tau H_2$. 
Each such group will be called a \emph{subdirect sum} of $G_1$ and $G_2$, and denoted briefly  $G_1 \oplus_\phi G_2$ (in this notation the normal subgroups $H_1$ and $H_2$ are assumed to be specified in the full description of the isomorphism $\phi$). 

If $H_1=G_1$, then $G_1/H_1$ is a trivial (one-element) group, and consequently, $G_2/H_2$ must be trivial, which means that we have also $H_2 = G_2$. Then, $G = G_1 \oplus G_2$ is the usual direct sum, with $\phi$ being the mapping between one-element sets. In such special case the subdirect sum $\oplus_\phi$ will be referred to as \emph{trivial}. 
In the case when $G_1 = G_2 = G$ and $H_1=H_2$ is the trivial one-element subgroup of $G$, and $\phi: G \to G$ is the identity mapping, the subdirect sum is nontrivial. In this case we call the resulting sum a \emph{parallel sum} (permutation isomorphic groups $G_1$ and $G_2$ act in a {parallel} manner on their orbits), and denote it briefly, by $2^\approx G$.
For example, the cyclic group generated by the permutation 
$\sigma = (1,2,3)(4,5,6)$ is permutation isomorphic to  $2^\approx C_3$. 

The main fact established in \cite{kis} is that every intransitive group has the form of a subdirect sum, and its components can be easily described. Let $G$ be an intransitive group acting on a set $X = X_1 \cup X_2$ in such a way that $X_1$ and $X_2$ are disjoint fixed blocks of $G$. Let $G_1$ and $G_2$ be restrictions of $G$ to the sets $X_1$ and $X_2$, respectively (they are called also \emph{constituents}). Let $H_1 \leq G_1$ and $H_2 \leq G_2$ be the subgroups fixing pointwise $X_2$ and $X_1$, respectively. Then we have

\begin{Theorem} {\rm \cite[Theorem 4.1]{kis}}
If $G$ is a permutation group as described above, then 
\begin{enumerate}
\item[\rm a)] $H_1$ and $H_2$ are normal subgroups of $G_1$ and $G_2$, respectively, 
\item[\rm b)] the factor groups $G_1/H_1$ and $G_2/H_2$ are abstractly isomorphic, and 
$$G = G_1[H_1] \oplus_\phi G_2[H_2],$$  
where $\phi$ is the isomorphism of the factor groups.
\end{enumerate}
\end{Theorem}

\section{Preliminary results} 

First, we establish the representability of \emph{regular} abelian permutation groups and some other groups connected with the automorphism groups of Cayley graphs. Here, we make use of known results on the so called \emph{Cayley index} of abelian groups.

Recall, that each regular permutation group may be viewed as the action of an abstract group $G$ on itself given by left multiplication. In such a case we have $X=G$, and the resulting permutation group will be denoted by $(G,G)$, or simply by $G$, if it is clear from the context that we mean the corresponding regular permutation group. In particular, we use standard notation $Z_k^m$ and $Z_k^m \times Z_r^s$ for abstract abelian groups to denote also corresponding permutation groups obtained by the regular action of these groups on themselves (in particular, $C_n$ and $Z_n$ denote here the same permutation group).
 
Given an abstract group $G$, by $\Cay(G)$ we denote the complete directed coloured Cayley graph, that is one with all nontrivial elements as generators defining different colours. Observe that  $\Cay(G) = \Orb(G,G)$.  
By $\Cay^*(G)$ we denote the complete undirected coloured graph obtained from $\Cay(G)$ by identifying colours corresponding to $g$ and $g^{-1}$ for every nontrivial $g\in G$, and removing the loops. Again, $\Cay^*(G) = \Orb^*(G,G)$. 

Now, given a set $S$ of nontrivial elements of $G$ (i.e., different from the identity), by $\Cay^*(G;S)$ we define the coloured graph obtained from $\Cay^*(G)$ by identifying all colours not in $S$. To admit further identifications, let $\Pi$ be a partition of $S$. Then by $\Cay^*(G;\Pi)$ we define the coloured graph obtained from $\Cay^*(G;S)$ by identifying the colours in each block of $\Pi$. In our notation applied below, $\Pi$ is written simply by listing its blocks, a block is written in the square brackets, and in the case of a one-element block, brackets are omitted. To make notation as compact as possible we adopt the convention that $\Pi$
contains only one representative of each pair $\{g,g^{-1}\}$. In addition, we assume that there are nontrivial pairs $g,g^{-1}$ not in $S$, and all elements not is $S$ get colour $0$ in diagrams represented by non-edges. 
Then, the graph $\Cay^*(G;\Pi)$ is the complete directed graph whose edges are coloured with exactly $|\Pi|+1$ colours.

The following lemma presents the coloured graphs whose automorphism groups are $(Z_2^k,Z_2^k)$ for $k=2,3,4$. (The $k$-tuples of elements of $Z_2^k$ are denoted below by corresponding strings of $0$'s and $1$'s).

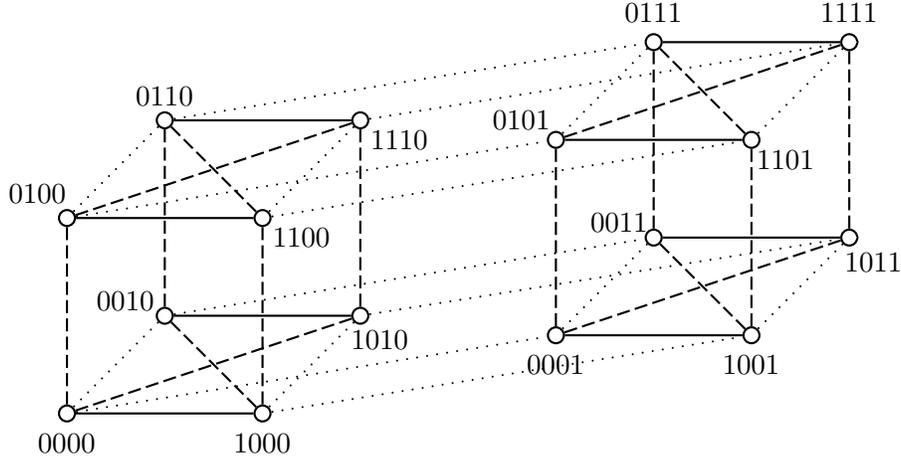
\begin{figure}
  \psset{unit=1.3cm, radius=0.12cm, arrowscale=1.8,arrowlength=0.8}
  \begin{center}
\begin{pspicture}(8,4.5)(0,-0.5)

\Cnode(0,0){v00}  
\Cnode(2,0){v10}  
\Cnode(0,2){v01} 
\Cnode(2,2){v11} 

\Cnode(1,1){w00}  
\Cnode(3,1){w10}  
\Cnode(1,3){w01} 
\Cnode(3,3){w11} 

\ncline{v00}{v10}
\ncline[linestyle=dashed, dash=5pt 2pt]{v00}{v01}
\ncline[linestyle=dashed, dash=5pt 2pt]{v11}{v10}
\ncline{v11}{v01}

\ncline{w00}{w10}
\ncline[linestyle=dashed, dash=5pt 2pt]{w00}{w01}
\ncline[linestyle=dashed, dash=5pt 2pt]{w11}{w10}
\ncline{w11}{w01}

\ncline[linestyle=dotted]{v00}{w00}
\ncline[linestyle=dotted]{v10}{w10}
\ncline[linestyle=dotted]{v01}{w01}
\ncline[linestyle=dotted]{v11}{w11}

\Cnode(5,0.8){dv00}  
\Cnode(7,0.8){dv10}  
\Cnode(5,2.8){dv01} 
\Cnode(7,2.8){dv11} 

\Cnode(6,1.8){dw00}  
\Cnode(8,1.8){dw10}  
\Cnode(6,3.8){dw01} 
\Cnode(8,3.8){dw11} 

\ncline{dv00}{dv10}
\ncline[linestyle=dashed, dash=5pt 2pt]{dv00}{dv01}
\ncline[linestyle=dashed, dash=5pt 2pt]{dv11}{dv10}
\ncline{dv11}{dv01}

\ncline{dw00}{dw10}
\ncline[linestyle=dashed, dash=5pt 2pt]{dw00}{dw01}
\ncline[linestyle=dashed, dash=5pt 2pt]{dw11}{dw10}
\ncline{dw11}{dw01}

\ncline[linestyle=dotted]{dv00}{dw00}
\ncline[linestyle=dotted]{dv10}{dw10}
\ncline[linestyle=dotted]{dv01}{dw01}
\ncline[linestyle=dotted]{dv11}{dw11}


\ncline[linestyle=dashed, dash=5pt 2pt]{v00}{w10}
\ncline[linestyle=dashed, dash=5pt 2pt]{v10}{w00}
\ncline[linestyle=dashed, dash=5pt 2pt]{v01}{w11}
\ncline[linestyle=dashed, dash=5pt 2pt]{v11}{w01}

\ncline[linestyle=dashed, dash=5pt 2pt]{dv00}{dw10}
\ncline[linestyle=dashed, dash=5pt 2pt]{dv10}{dw00}
\ncline[linestyle=dashed, dash=5pt 2pt]{dv01}{dw11}
\ncline[linestyle=dashed, dash=5pt 2pt]{dv11}{dw01}

\ncline[linestyle=dotted]{v00}{dv00}
\ncline[linestyle=dotted]{v10}{dv10}
\ncline[linestyle=dotted]{v01}{dv01}
\ncline[linestyle=dotted]{v11}{dv11}

\ncline[linestyle=dotted]{w00}{dw00}
\ncline[linestyle=dotted]{w10}{dw10}
\ncline[linestyle=dotted]{w01}{dw01}
\ncline[linestyle=dotted]{w11}{dw11}

\rput(-0,-0.3){0000} \rput(-0.3,2.25){0100}
\rput(2,-0.3){1000}
\rput(1,3.25){0110}

\rput(6,4.1){0111} \rput(8,4.1){1111}
\rput(8.25,1.55){1011}
\rput(7,0.5){1001}

\rput(0.6,1.15){0010}
\rput(5,0.5){0001}

\rput(3.2,0.75){1010} \rput(2.4,1.8){1100}
\rput(3.4,2.8){1110}

\rput(5.65,1.95){0011} \rput(4.65,3){0101}
\rput(7.35,2.6){1101}



\end{pspicture}
\end{center} \caption{$\Cay^*(Z_2^4; 1000, [0100, 1010], [0010,  0001])$}\label{fig1}
\end{figure}

\begin{Lemma}\label{lem:z234}
Each of the following coloured graphs represents the regular action of its defining group:
\begin{enumerate}
\item[(i)] $\Cay^*(Z_2^2; 10, 01)$, 
\item[(ii)] $\Cay^*(Z_2^3; 100, 010, 001)$, 
\item[(iii)] $\Cay^*(Z_2^4; 1000, [0100, 1010], [0010,  0001])$
\end{enumerate}

\end{Lemma}

\begin{proof}
We consider the case (iii). Denote the graph by $\Gamma$. It is pictured in Figure~\ref{fig1}. (Solid, dashed, and dotted lines correspond to colours $1000,0100$, and $0010$, respectively). We will speak correspondingly of solid, dashed, and dotted neighbors.

Since $\Gamma$ is obtained from a Cayley graph on $Z_2^4$ by identifying colours, it follows that its automorphism group $\Aut(\Gamma)$ contains the regular action of $Z_2^4$ (which in this notation is given by addition). We need only to prove that $\Aut(\Gamma)$ contains no other permutation. Let us consider the stabilizer $A_0$ of the vertex $0000$ in $\Aut(\Gamma)$. As the latter is transitive, it is enough to show that $A_0$ is trivial.

Since the only solid neighbor of $0000$ is $1000$, $A_0$ fixes $1000$ as well.  
Further, the only dashed neighbor of $0000$ that is a dotted neighbor of $1000$ is $1010$, while the only solid neighbor of the latter is $0010$. 
Thus the four vertices with coordinates $x0y0$ are fixed.
Considering their dashed neighbors, we see that also each vertex with coordinates $x1y0$ must be fixed. It follows that the vertices of the cube $xyz0$ are individually fixed. Considering their dotted neighbors, we conclude that the same holds for the cube $xyz1$, which completes the proof.

The cases (i) and (ii) are easier, and are left to the reader as an exercise. 
\end{proof}

Now, recall that an abelian permutation group $A = (A,X)$ is transitive if and only if it is regular. It follows that 
transitive abelian permutation groups $A$ can be identified with the regular action of $A$ (considered as an abstract group) on itself.
In this case, we have a special permutation on $A$ defined by $\alpha : x\to x^{-1}$ called the \emph{involution}. (For  properties and a very special role of this permutation see, e.g., \cite{DSV,IW}). It is easy to observe that the involution preserves the colours of the edges in $\Cay^*(A)$. This leads to the well-known fact:

\begin{Lemma}\label{lem:alpha}
Le $A$ be a regular abelian permutation group, and $\alpha$ its involution. If $\Gamma$ is a coloured graph such that $\Aut(\Gamma) \supseteq A$, then $\alpha \in \Aut(\Gamma)$.
\end{Lemma}

This is so since $\Cay^*(A) = \Orb(A,A)$, and $\Gamma$ needs to be a graph obtained from $\Orb(A,A)$ by identification of colours.

It follows from this lemma that generally a regular abelian permutation group $A$ does not belong to $GR$, except for the case when $\alpha$ is trivial (the identity permutation). This is exactly the case, when  $A=Z_2^n$ for some $n\geq 0$. It is well known that for $n\geq 5$, $Z_2^n$ is representable as the automorphism group of a simple (Cayley) graph (see \cite{imr}, or claim 1.2 in \cite{IW}). 
Combining this with Lemma~\ref{lem:z234} we have

\begin{Lemma}\label{lem:z2n}
Let $A$ be a regular abelian permutation group. If $A=Z_2^n$ for some $n\geq 0$, then $A\in GR(4)$; otherwise, $A\notin GR$
\end{Lemma}

We note that $Z_2^3$ requires $4$ colours, in the sense that there exists no $k$-coloured graph with $k<4$ whose automorphism group is $Z_2^3$. The proof of this fact is rather tedious, but one may also check this with a help of computer. We mention it, because it means that the number $4$ in the results of this paper cannot be lowered.

The permutation group generated by left translations of a regular abelian group $A$ and its involution $\alpha$ plays a special role in this paper. We denote it by $A^+ = \<A,\alpha\>$. We note that if $\alpha$ is nontrivial, then $A^+$ is nonabelian. Nevertheless we need knowledge on representability of such groups, and to establish it, we apply Theorem~1 in~\cite{IW}.

\begin{Lemma}\label{lem:Aalpha}
If $A$ is a regular abelian permutation group, than $A^+ \in GR(2)$, except for the following groups: $A = Z_2^2, Z_2^3, Z_2^4, Z_4\times Z_2, Z_4\times Z_2^2, Z_3^2, Z_3^3$, or $Z_4^2$.  In any case, $A^+\in GR(4)$.
\end{Lemma}

\begin{proof}
The first claim follows from \cite[Theorem~1]{IW} combined with the remark 1.2 preceding this theorem (which adds to the list of exceptions $Z_2^2$).  For the second claim, the three first cases follow by Lemma~\ref{lem:z234}. For the remaining cases the following five $4$-coloured graphs of the form $\Cay^*(A,P)$ have, in any case, the automorphism group equal to  $A^+$: 
 
$\Cay^*(Z_4\times Z_2; 10, 01)$,
$\Cay^*(Z_4\times Z_2^2; 001,011,[100,010])$,

$\Cay^*(Z_3^2; 10, 01,11)$, 
$\Cay^*(Z_3^3; 010,[001,100],[110,101])$, 

$\Cay^*(Z_4^2; 10,01,13).$ 

Checking this fact for each of the five graphs is routine, and similar to checking the case (iii) in the proof of Lemma~\ref{lem:z234}. 
\end{proof}

We note that $\Cay^*(Z_3^2; 10, 01,11)$ (pictured in the left hand side of Figure~\ref{fig2}) is a \emph{unique $4$-coloured graph} with the unique automorphism group $(Z_3^2)^+$, which means that no other coloured graph has such an automorphism group (see \cite{GK1}). In particular,  
the number $4$ in this lemma cannot be lowered.

We have also two exceptional intransitive abelian permutation group whose representability needs to be established directly. They are two nontrivial subgroups of the direct sum $Z_3^2\oplus Z_3^2$.

\begin{Lemma}\label{lem:z32z32}
Let $A$ be a nontrivial subgroup of $Z_3^2\oplus Z_3^2$ of the form
$$A = Z_3^2[H] \oplus_\phi  Z_3^2[H],$$
such that $H=Z_3$ or $H=I_9$. Then 
$A\in GR(4)$.
\end{Lemma}
\begin{proof}
First consider the case when $H$ is a subgroup isomorphic to $Z_3$. Note 
that in this case the decomposition formula above describes $A$ uniquely up to permutation isomorphism. Indeed, each subgroup $Z_3$ of $Z_3^2$ may be treated as one of the summands of suitably presented $Z_3^2$, and the isomorphism $\phi$ between groups isomorphic to $Z_3$ is unique up to renaming generators of $Z_3$. 

We construct a graph $\Gamma$ as a suitable composition of two graphs of the form $\Cay^*(Z_2^3;\Pi)$. 
The first component of $\Gamma$ (corresponding to the first orbit of $A$) is $\Cay^*(Z_2^3;01,11,10)$, and the second one (corresponding to the second orbit of $A$) is  $\Cay^*(Z_2^3;10,01)$ (e assume here that the colours $1,2,3$ are assigned to edges in accordance with the position on the list, so in particular,  the edge $(00,01)$ in the first graph has the same colour $1$ as the edge $(00,10)$ in the second graph). To describe the colours of edges between the two components, we assume that the pairs in the second component are denoted with overline. (Thus, the first component consist of pairs $xy$, where $x,y \in \{0,1,2\}$, and the second component consists of analogous overlined pairs $\overline{xy}$). Then, we put the colour 3 for the edge $(00,\overline{00})$ and the colour 1 for the edge $(00,\overline{01})$. This is done under assumption that $\Aut(\Gamma) \supseteq A$, where the first $Z_3$ subgroup in the decomposition $A = Z_3^2[Z_3] \oplus_\phi  Z_3^2[Z_3]$ is equal to $(Z_3,\{00,10,20\})$, and the second to $(Z_3,\{\overline{00},\overline{10},\overline{20}\})$. This 
assumption forces the colours for other edges in those orbitals of $A$ that contain the mentioned edges. The remaining edges are coloured $0$. Note, that here is no edge of colour $2$ between the components. 

The graph $\Gamma$ is illustrated in Figure~\ref{fig3}. The dashed, dotted, and solid lines correspond to colours $1,2$, and $3$, respectively. To make the drawing  more readable, we have applied the convention that each line between components ending with double arrows correspond to nine edges in the given colour joining each vertex in the horizontal line pointed out by the arrows in the left component with each vertex in the horizontal line pointed out by the arrows in the right component.

\begin{figure}
  \psset{unit=1.8cm, radius=0.12cm, arrowscale=1.8,arrowlength=0.8}
  \begin{center}
\begin{pspicture}(6,3.5)(0,0.5)

\rput(0.25,0.75){$00$}
\rput(1.25,0.7){$10$}
\rput(2.25,0.75){$20$}

\rput(3.65,0.75){$\overline{00}$}
\rput(4.65,0.7){$\overline{10}$}
\rput(5.65,0.75){$\overline{20}$}

\rput(0.25,3.2){$02$}
\rput(1.25,3.2){$12$}
\rput(2.25,3.2){$22$}

\rput(0.35, 1.8){$01$}
\rput(1.35,1.75){$11$}
\rput(2.35,1.8){$21$}

\rput(3.55,2.2){$\overline{01}$}
\rput(4.55,2.2){$\overline{11}$}
\rput(5.55,2.2){$\overline{21}$}

\rput(3.65,3.2){$\overline{02}$}
\rput(4.65,3.2){$\overline{12}$}
\rput(5.65,3.2){$\overline{22}$}

\Cnode(0.2,1){v00}  
\Cnode(0.2,2){v01}  
\Cnode(0.2,3){v02} 

\Cnode(1.2,1){v10}  
\Cnode(1.2,2){v11}  
\Cnode(1.2,3){v12} 

\Cnode(2.2,1){v20}  
\Cnode(2.2,2){v21}  
\Cnode(2.2,3){v22} 

\ncline{v00}{v10}
\ncline{v10}{v20}
\ncarc[arcangle=-15]{v00}{v20}

\ncline{v01}{v11}
\ncline{v11}{v21}
\ncarc[arcangle=-15]{v01}{v21}

\ncline{v02}{v12}
\ncline{v12}{v22}
\ncarc[arcangle=-15]{v02}{v22}

\ncline[linestyle=dashed, dash=5pt 2pt]{v00}{v01}
\ncline[linestyle=dashed, dash=5pt 2pt]{v01}{v02}
\ncarc[linestyle=dashed, dash=5pt 2pt,arcangle=15]{v00}{v02}

\ncline[linestyle=dashed, dash=5pt 2pt]{v10}{v11}
\ncline[linestyle=dashed, dash=5pt 2pt]{v11}{v12}
\ncarc[linestyle=dashed, dash=5pt 2pt,arcangle=15]{v10}{v12}

\ncline[linestyle=dashed, dash=5pt 2pt]{v20}{v21}
\ncline[linestyle=dashed, dash=5pt 2pt]{v21}{v22}
\ncarc[linestyle=dashed, dash=5pt 2pt,arcangle=15]{v20}{v22}

\ncline[linestyle=dotted]{v00}{v11}
\ncline[linestyle=dotted]{v11}{v22}
\ncarc[linestyle=dotted, arcangle=-10]{v00}{v22}

\ncline[linestyle=dotted]{v02}{v21}
\ncline[linestyle=dotted]{v02}{v10}
\ncline[linestyle=dotted]{v10}{v21}

\ncline[linestyle=dotted]{v20}{v12}
\ncline[linestyle=dotted]{v20}{v01}
\ncline[linestyle=dotted]{v01}{v12}


\Cnode(3.7,1){w00}  
\Cnode(3.7,2){w01}  
\Cnode(3.7,3){w02} 

\Cnode(4.7,1){w10}  
\Cnode(4.7,2){w11}  
\Cnode(4.7,3){w12} 

\Cnode(5.7,1){w20}  
\Cnode(5.7,2){w21}  
\Cnode(5.7,3){w22} 

\ncline[linestyle=dashed, dash=5pt 2pt]{w00}{w10}
\ncline[linestyle=dashed, dash=5pt 2pt]{w10}{w20}
\ncarc[linestyle=dashed, dash=5pt 2pt,arcangle=-15]{w00}{w20}

\ncline[linestyle=dashed, dash=5pt 2pt]{w01}{w11}
\ncline[linestyle=dashed, dash=5pt 2pt]{w11}{w21}
\ncarc[linestyle=dashed, dash=5pt 2pt,arcangle=-15]{w01}{w21}

\ncline[linestyle=dashed, dash=5pt 2pt]{w02}{w12}
\ncline[linestyle=dashed, dash=5pt 2pt]{w12}{w22}
\ncarc[linestyle=dashed, dash=5pt 2pt,arcangle=-15]{w02}{w22}

\ncline[linestyle=dotted]{w00}{w01}
\ncline[linestyle=dotted]{w01}{w02}
\ncarc[linestyle=dotted,arcangle=-15]{w00}{w02}

\ncline[linestyle=dotted]{w10}{w11}
\ncline[linestyle=dotted]{w11}{w12}
\ncarc[linestyle=dotted,arcangle=-15]{w10}{w12}

\ncline[linestyle=dotted]{w20}{w21}
\ncline[linestyle=dotted]{w21}{w22}
\ncarc[linestyle=dotted,arcangle=-15]{w20}{w22}

\ncline[nodesep=3pt]{<<->>}{v22}{w02}
\ncline[nodesep=3pt]{<<->>}{v21}{w01}
\ncline[nodesep=3pt]{<<->>}{v20}{w00}

\ncline[linestyle=dashed, dash=5pt 2pt, nodesep=5pt]{<<->>}{v22}{w00}
\ncline[linestyle=dashed, dash=5pt 2pt, nodesep=5pt]{<<->>}{v20}{w01}
\ncline[linestyle=dashed, dash=5pt 2pt, nodesep=5pt]{<<->>}{v21}{w02}

\end{pspicture}
\end{center} \caption{$\Aut(\Gamma) = Z_3^2[Z_3] \oplus_\phi  Z_3^2[Z_3]$}\label{fig3}
\end{figure}
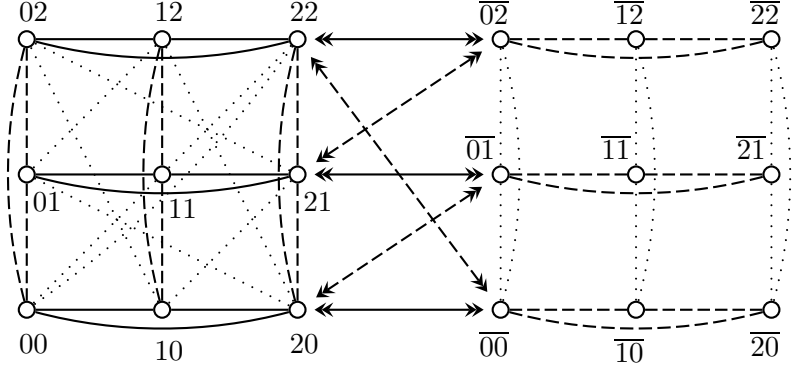

Thus we have  $\Aut(\Gamma) \supseteq A=Z_3^2[Z_3] \oplus_\phi  Z_3^2[Z_3]$, where $\phi$ is the natural isomorphism between $Z_3$-subgroups. 
Moreover,  $\Aut(\Gamma)$ preserves the orbits of $A$, since the quadruples of colour degrees in the first component is different from that in the second component. 

Now, consider the stabilizer $S_0$ of point $00$ in $\Aut(\Gamma)$. To get the equality above, it is enough to show that the cardinality $|S_0|=3$. Because of the dotted edges (colour $2$) coming out from $00$, $S_0$ fixes the set $\{11,22\}$. We show that, actually, $S_0$ fixes individually points in this pair. Suppose, to the contrary, that $S_0$ contains an automorphism $\sigma$ transposing $11$ and $22$.

Then, $\sigma$ transposes sets $\{10,11,12\}$ and  $\{20,21,22\}$, which are disjoint triangles in colour $3$ (solid line) in the first component. 
Due to edges in colour $3$ between the components, $\sigma$ transposes also the sets: 
$\{\overline{01},\overline{11},\overline{21}\}$ and  $\{\overline{02},\overline{12},\overline{22}\}$, which are the triangles in colour $1$  (dashed lines) in the second component. Similarly, due to edges in colour $1$ between the components, $\sigma$ transposes the sets: 
$\{\overline{00},\overline{10},\overline{20}\}$ and  $\{\overline{02},\overline{12},\overline{22}\}$, which contradicts the previous claim. 

Thus, the set $\{00,11,12\}$ remains individually fixed in $S_0$. It is now easy to see that each point in the first component remains fixed under $S_0$.  Consequently, (because of solid edges  between the two components) each triangle in colour $1$ in the second component remains fixed, and (due to the edges coloured $2$ in the second component), for every automorphism, the permutation in one of this triangles determines permutations in other triangles.  Thus, the cardinality of the stabilizer of $00$ in $\Aut(\Gamma)$ is equal to $3$, proving the lemma for $H=Z_3$.

Consider now the case when $H=I_9$ is the trivial subgroup of $Z_3^2$. Then $A$ is the parallel sum  
$A = 2^\approx Z_3^2$.  
In this case we construct a graph $\Gamma$ as a composition of $\Cay^*(Z_2^3;10,01,11)$ and  $\Cay^*(Z_2^3;01,10)$. The edges of colour $1$ join each pair $(xy,\overline{xy})$, and the other edges are coloured $0$. The proof that $\Aut(\Gamma) = 2^\approx Z_3^2$ is similar as in the first case (but simpler), so we leave it to the reader. 
\end{proof}

In fact, the parallel sum $2^\approx Z_3^2$ is known to belong to $GR(2)$. A suitable construction is contained in the proof of the main result in \cite{bab}.

\section{The structure of abelian permutation groups}

From now on $A = (A,X)$ denotes an abelian permutation group on a fixed set $X$, with orbits $X_1,\ldots , X_r$. Then, by $A_i=A|_{X_i}$ we denote the restriction of $A$ to $X_i$, by $A^j_i$ the restriction of the pointwise stabilizer of the orbit $X_j$ to the orbit $X_i$, and by $A^*_i$  the restriction of the pointwise stabilizer of the set $V\setminus X_i$ to the orbit $X_i$. 

Two orbits $X_i$ and $X_j$, $i\neq j$, are called \emph{adjacent} if the factor group $A_i/A_i^j$ is not an elementary abelian $2$-group. We note that this relation is symmetric. Indeed, the restriction $B$ of $A$ to $X_i\cup X_j$ can be presented as $B= A_i[A_i^j] \oplus_\phi A_j[A_j^i]$. This means, in particular, that  $A_i/A_i^j$ is isomorphic to  $A_j/A_j^i$, which implies the claim. 
Accordingly, an orbit $X_i$ of $A$ will be called \emph{isolated}, if it is not adjacent to any orbit $X_j$, 
$j\neq i$. 

Let us recall, that a permutation $\sigma$ preserving orbits of $A$ is called \emph{$2$-orbit-compatible} with the permutation group $A$,  if for each pair of orbits $X_i$ and $X_j$, $i\neq j$, the restriction of $\sigma$ to $X_i\cup X_j$ belongs to the restriction of the group $A$ to $X_i\cup X_j$. The group $A$ is $2$-\emph{orbit-closed} if every permutation that is $2$-{orbit-compatible} with $A$ belongs to $A$. The $2$-{orbit-closure} of $A$, denoted $\clo{A}$, is the group consisting of all permutations $2$-{orbit-compatible} with $A$. 
Obviously, $\clo{A}$ has the same orbits as $A$. Moreover, it has the same components $A_i$, $A_i^j$, $A_i^*$.  In particular, we have the following.

\begin{Lemma}\label{lem:clo}
Let $A$ be a permutation group. Then the following hold
\begin{enumerate}
\item  $A$ is abelian if and only if $\clo{A}$ is abelian.
\item An orbit $X_i$ is isolated in $A$ if and only if $X_i$ is isolated in $\clo{A}$.
\end{enumerate}
\end{Lemma}

The notion of $2$-orbit closure arise naturally, when one considers automorphism groups of coloured graphs and digraphs. All these groups are obviously $2$-orbit closed. It is enough to observe that a coloured graph (or digraph) $\Gamma$ has exactly two kinds of edges with regard to its automorphism group: those joining vertices within an orbit of the group and those joining vertices between two different orbits. It is easily seen that a permutation $2$-orbit-compatible with $\Aut(\Gamma)$ preserves the colours of all edges. The following is an obvious property of $2$-\emph{orbit-closed} groups.

\begin{Lemma}\label{lem:2equal}
Let $A$ and $B$ be $2$-orbit-closed permutation groups acting on the same set $X$ and having the same orbits. If for any two orbits $O$ and $Q$, the restriction $A|_{O\cup Q} = B|_{O\cup Q}$, then $A=B$.
\end{Lemma}

We have also the following crucial characterization

\begin{Lemma}\label{lem:isol}
If $A$ is $2$-orbit-closed, then for every orbit $X_i$ of $A$, $X_i$ is isolated in $A$ if and only if 
$A_i/A_i^*$ is an elementary abelian $2$-group.
\end{Lemma}
\begin{proof} 
First observe that for each $i\leq r$, $A_i^* \subseteq  \bigcap_{j\neq i} A_i^j$. We show that for $2$-orbit-closed groups the converse inclusion holds, as well. Indeed, suppose that $\tau\in A_i^j$ for each $j\neq i$. It follows, that for each $j\neq i$, there is a permutation $\sigma_j \in A$, such that its restriction to $X_i\cup X_j$ is equal to $\tau$ extended to $X_j$ by fixing all points in $X_j$. Consequently, the permutation $\sigma$ equal to $\tau$ and fixing all points in $X\setminus X_i$ is $2$-orbit-compatible with $A$, and therefore belongs to $A$. Whence, $\tau\in A_i^*$, as required.

Now, we prove our claim by contraposition. Suppose that $A_i/A_i^*$ is not an elementary abelian $2$-group, that is, it has an element $xA_i^*$ of order $> 2$. This is equivalent to that $x^2 \notin A_i^*$, which means (by what proved above) that there is $j\neq i$ such that $x^2 \notin A_i^j$. The latter is equivalent to that  $A_i/A_i^j$  has an element $xA_i^*$ of order $> 2$, that is, $A_i/A_i^j$  is not an elementary abelian $2$-group. This means that $X_i$ is not isolated, and the result follows.
\end{proof}

Note that the factor group $A_i/A_i^*$ is an abstract group; we do not define any action of this group.  It plays a special role in our main result below, and we will refer to it as to \emph{the factor group of the orbit} $X_i$.

\section{Characterization of $2^*$-closed abelian permutation groups}

We state our main result

\begin{Theorem}\label{th:main}
Let $A$ be a nontrivial abelian permutation group. Then $A$ is the automorphism group of a coloured graph if and only if the following conditions hold
\begin{enumerate}
\item $A$ is $2$-orbit-closed, and 
\item for every orbit $X_i$ of $A$, if the factor group $A_i/A_i^*$  is an elementary abelian $2$-group, then so is $A_i$. 
\end{enumerate} 
\end{Theorem}

Note that the factor group $A_i/A_i^*$  is an elementary abelian $2$-group (as an abstract group) if and only if it is isomorphic to $Z_2^n$ for some $n\geq 0$. In turn, the permutation group $A_i$ is an elementary abelian $2$-group if and only if it is permutation isomorphic to the regular action of $Z_2^n$ for some $n\geq 0$. Note that this includes trivial cases with $n=0$. (There are also other permutation groups that are elementary abelian $2$-groups, but they are not transitive).

The proof of Theorem~\ref{th:main} consists of a number of lemmas. We keep the notation of the previous section. First we prove the ``only if'' part of the theorem. 

\begin{Lemma}\label{lem:oifpart}
If an abelian permutation group $A\in GR$, then $A$ satisfies the conditions $(1)$ and $(2)$ of the Theorem~5.1.
\end{Lemma}

\begin{proof}
As we have already noted before Lemma~\ref{lem:2equal}, the condition (1)  obviously holds. For (2), let $\Gamma$ be a coloured graph with $\Aut(\Gamma)= A$, and suppose that $A_i/A_i^*$ is isomorphic to $Z_2^m$ for some $m\geq 0$. Since $A_i$ is abelian and transitive on $X_i$, it acts regularly on $X_i$. Therefore $X_i$ may be identified with $A_i$, and the action of $A_i$ with the regular action on itself. In particular, $A_i^*$ may be considered as a subset of $X_i$.

For each pair of elements $x,y\in A_i^*$, there is a permutation $\sigma\in A$  moving $x$ into $y$ and fixing all the elements outside $X_i$. Because of commutativity, $\sigma$ does the same with any pair $tx$ and $ty$, where $t\in A_i$ is treated as a permutation on $A_i$. It follows that
the cosets of $A_i/A_i^*$ have the same property:  for each pair of elements $x,y$ in the same coset, there is a permutation $\sigma\in A$  moving $x$ into $y$ and fixing all the elements outside $X_i$.  It follows that for every pair of such elements $x,y$, and every element $z\notin X_i$, the edges $xz$ and $yz$ in $\Gamma$ have the same colour.

We observe that for each $x\in X_i$, $x^{-1}$ is in the same coset as $x$.
Indeed, since $A_i/A_i^* \cong Z_2^m$, for cosets we have $xA_i^*xA_i^*= A_i^*$, and by commutativity, $x^2A_i^*= A_i^*$; hence $xA_i^* = x^{-1}A_i^*$, as required. Thus, we infer that the edges $xz$ and $x^{-1}z$ in $\Gamma$ have the same colour, for every element $z\notin X_i$. 

We proceed to show the the involution $\alpha$ in $X_i=A_i$ treated as a permutation of $X$ (fixing all elements $x\notin X_i$) preserves the colours of edges in $\Gamma$. Indeed, by what established above, it preserves the colours of all edges in $\Gamma$ that have at most one end in $X_i$. On the other hand, by Lemma~\ref{lem:alpha} we know that $\alpha$ preserves the colours
of the edges within $X_i$, which proves the claim.

Consequently, $\alpha \in A_i$. Since $A_i$ is regular, it means that $\alpha$ must be trivial, that is, $x=x^{-1}$ for all $x \in A_i$. It follows that $A_i$ is an elementary abelian $2$-group, proving the lemma.
\end{proof}

The proof of the ``if'' part is by induction on the number of orbits in $A$. Below, we establish the result for two orbits. Note that in this case $A$ is trivially $2$-orbit-closed.

\begin{Lemma}
\label{lem:2orb}
If $A$ is a nontrivial abelian permutation group with two orbits and $A$ satisfies condition $(2)$ of the theorem,  then $A\in GR(4)$. 
\end{Lemma}

\begin{proof} 

Let $O$ and $Q$ be the orbits of $A$. 
Let $A = B[B'] \oplus_\phi C[C']$ be the subdirect decomposition of $A$ with regard to $O$ and $Q$. 
Let $O_1, \ldots, O_r$ be the partition of the orbit $O$  into orbits of $B'$. Since the actions of $B$ on $O$ is regular (as $B$ is abelian and transitive), the factor group $B/B'$ acts on the set of orbits $O_1, \ldots, O_r$  in a regular way. The same is true of the action of $C/C'$ on the set of orbits $Q_1,\ldots, Q_s$ of $C'$ on $Q$, and since $B/B'$ and $C/C'$ are isomorphic, $r=s$. It follows also that $\phi : B/B' \to C/C'$ establishes a one to one correspondence between the orbits $O_1, \ldots, O_r$ and $Q_1,\ldots, Q_r$ so that the action of $B/B'$ on $O_1, \ldots, O_r$ is equivalent to the action of $C/C'$ on $Q_1,\ldots, Q_r$. After suitable renumbering we may assume that $\phi(O_i) = Q_i$ for all $i=1,\ldots,r$. Moreover, we may assume that the orbits $O_i$ are identified with the cosets of $B/B'$, the orbits $Q_i$ are identified with the cosets of $C/C'$, and $B' = O_1$ and $C' = Q_1$. 

We construct a 4-coloured graph on the set $O\cup Q$ of vertices. It consists of two parts: $\Gamma_1$ on the set of vertices $O$, and $\Gamma_2$ on the set of vertices $Q$. We take $\Gamma_i$ to be a $4$-coloured graph given by Lemma~\ref{lem:Aalpha}, such that $\Aut(\Gamma_1) = \langle B,\beta \rangle$ and $\Aut(\Gamma_2) = \langle C,\gamma \rangle$, where $\beta$ and $\gamma$ are corresponding involutions.

For the edges joining $O$ and $Q$ we put colours as follows. First, for each $i=1,\ldots,r$, and for all $y\in O_i$ and all $z\in Q_i$, the edge $yz$ is coloured $1$. These edges reflect the one-to-one correspondence between cosets. They guarantee that the regular action on cosets is \emph{parallel}: if $(\sigma, \tau)$ is a permutation on $O\cup Q$ preserving the set of these edges (where $\sigma$ permutes $O$, and $\tau$ permutes $Q$), then $\sigma(O_i)= O_j$ implies $\tau(Q_i)= Q_j$ for all $i,j \leq r$. Therefore, in the remaining part of the construction we assume that these edges are the only edges between $O$ and $Q$ coloured $1$. Note that this guarantees also that, if $\Aut(\Gamma) \subseteq B \oplus C$, than  $\Aut(\Gamma) = B[B'] \oplus_\phi C[C'] = A$. So, it remains only to prove that $\Aut(\Gamma) \subseteq B \oplus C$. (This construction will be referred further as \emph{joining cosets in parallel manner}).

Now the construction differs depending on whether the orbits $O$ and $Q$ are adjacent or not. First we consider the case of adjacent orbits, and define the set of edges between $O$ and $Q$ coloured $2$. They are chosen to prevent involutions in $O$ and $Q$. 

Since $B/B'$ is not isomorphic to $Z_2^n$, it has an element $xB'$ of order greater than $2$. If we would have $xB' = x^{-1}B'$, then $x^2B' = B'$, a contradiction. Hence $x$ and $x^{-1}$ lie in different cosets. 
We colour all the edges between $B$ and $\phi(xB)$ with the colour $2$. Moreover, to make sure that $\Aut(\Gamma) \supseteq A$, we put colour $2$ for all edges between $yB$ and $\phi(yxB)$ for any $y\in B$. The remaining edges between $X$ and $O$ are coloured $0$.  Thus, since $x^{-1}B \neq xB$, the edges between  $B$ and $\phi(x^{-1}B)$ have colour $0$, while the edges between  $B$ and $\phi(xB)$ have colour $2$. This ensures that the involution $\gamma$ does not preserve colours of the edges, and similarly, $\beta$ does not, either.

\begin{figure}
  \psset{unit=2cm, radius=0.1cm} 

\begin{center}
\begin{pspicture}(3.5,3.8)(1.4,0)

\Cnode(4,0.1){v11}  
\Cnode(4,0.3){v12}  
\Cnode(4,0.5){v13} 

\Cnode(4,1.1){v21}  
\Cnode(4,1.3){v22}  
\Cnode(4,1.5){v23} 

\Cnode(4,2.1){v31}  
\Cnode(4,2.3){v32}  
\Cnode(4,2.5){v33} 

\Cnode(4,3.1){v41}  
\Cnode(4,3.3){v42}  
\Cnode(4,3.5){v43}





\ncarc[arcangle=30]{v11}{v21}
\ncarc[arcangle=30]{v21}{v31}
\ncarc[arcangle=30]{v31}{v41}
\ncarc[arcangle=40]{v41}{v12}

\ncarc[arcangle=-30]{v12}{v22}
\ncarc[arcangle=-30]{v22}{v32}
\ncarc[arcangle=-30]{v32}{v42}
\ncarc[arcangle=40]{v42}{v13}

\ncarc[arcangle=30]{v13}{v23}
\ncarc[arcangle=30]{v23}{v33}
\ncarc[arcangle=30]{v33}{v43}
\ncarc[arcangle=60]{v43}{v11}

\Cnode(1,0.15){w11}  
\Cnode(1,0.45){w12}  
 
\Cnode(1,1.15){w21}  
\Cnode(1,1.45){w22}  
 
\Cnode(1,2.15){w31}  
\Cnode(1,2.45){w32}  

\Cnode(1,3.15){w41}  
\Cnode(1,3.45){w42}

\rput(0.85,0.1){0}  
\rput(1,0.62){4} 

\rput(0.85,1.15){1}  
\rput(0.95,1.62){5}  
 
\rput(0.85,2.15){2}  
\rput(0.95,2.62){6}  

\rput(0.85,3.15){3}  
\rput(0.95,3.62){7}

\ncline{v11}{w11}
\ncline{v11}{w12}
\ncline{v12}{w11}
\ncline{v12}{w12}
\ncline{v13}{w11}
\ncline{v13}{w12}

\ncline{v21}{w21}
\ncline{v21}{w22}
\ncline{v22}{w21}
\ncline{v22}{w22}
\ncline{v23}{w21}
\ncline{v23}{w22}

\ncline{v31}{w31}
\ncline{v31}{w32}
\ncline{v32}{w31}
\ncline{v32}{w32}
\ncline{v33}{w31}
\ncline{v33}{w32}

\ncline{v41}{w41}
\ncline{v41}{w42}
\ncline{v42}{w41}
\ncline{v42}{w42}
\ncline{v43}{w41}
\ncline{v43}{w42}

 \psset{linewidth=0.021}
\ncarc[arcangle=40,linestyle=dotted]{w11}{w21}
\ncarc[arcangle=40,linestyle=dotted]{w21}{w31}
\ncarc[arcangle=40,linestyle=dotted]{w31}{w41}
\ncarc[arcangle=-40,linestyle=dotted]{w41}{w12}

\ncarc[arcangle=-40,linestyle=dotted]{w12}{w22}
\ncarc[arcangle=-40,linestyle=dotted]{w22}{w32}
\ncarc[arcangle=-40,linestyle=dotted]{w32}{w42}
\ncarc[arcangle=-60,linestyle=dotted]{w42}{w11}

\ncline[linestyle=dotted]{v11}{w41}
\ncline[linestyle=dotted]{v11}{w42}
\ncline[linestyle=dotted]{v12}{w41}
\ncline[linestyle=dotted]{v12}{w42}
\ncline[linestyle=dotted]{v13}{w41}
\ncline[linestyle=dotted]{v13}{w42}

\ncline[linestyle=dotted]{v21}{w11}
\ncline[linestyle=dotted]{v21}{w12}
\ncline[linestyle=dotted]{v22}{w11}
\ncline[linestyle=dotted]{v22}{w12}
\ncline[linestyle=dotted]{v23}{w11}
\ncline[linestyle=dotted]{v23}{w12}

\ncline[linestyle=dotted]{v31}{w21}
\ncline[linestyle=dotted]{v31}{w22}
\ncline[linestyle=dotted]{v32}{w21}
\ncline[linestyle=dotted]{v32}{w22}
\ncline[linestyle=dotted]{v33}{w21}
\ncline[linestyle=dotted]{v33}{w22}

\ncline[linestyle=dotted]{v41}{w31}
\ncline[linestyle=dotted]{v41}{w32}
\ncline[linestyle=dotted]{v42}{w31}
\ncline[linestyle=dotted]{v42}{w32}
\ncline[linestyle=dotted]{v43}{w31}
\ncline[linestyle=dotted]{v43}{w32}

\end{pspicture}
\end{center} \caption{Graph $\Gamma$ with $\Aut(\Gamma) = Z_{8}[Z_2]\oplus_\phi Z_{12}[Z_3]$, consisting of two cycles representing $Z_{8}$ and $Z_{12}$, and edges joining the cycles.  }\label{fig4}
\end{figure}
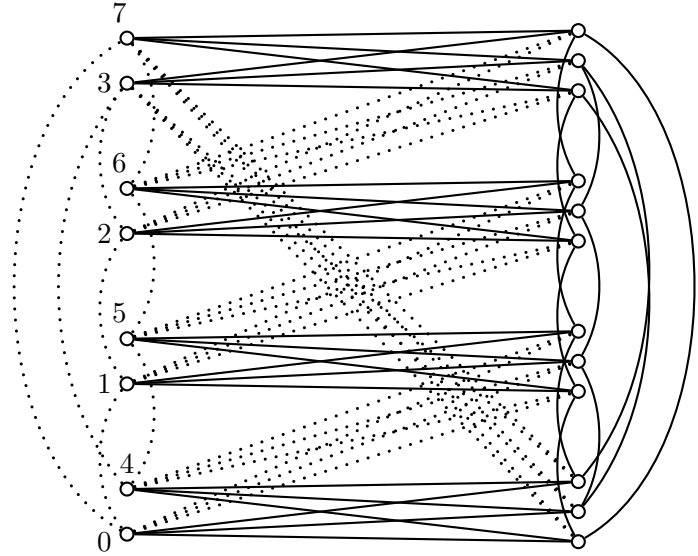

(Figure~\ref{fig4} illustrates the case for $B=Z_8, B'=Z_2, C=Z_{12}$, and $C'=Z_3$; solid lines correspond to colour $1$, while dotted lines correspond to colour $2$. The cycles representing $(Z_8)^+$ (on the left, in colour $2$), and $(Z_{12})^+$ (on the right, in colour $1$) are drawn in a way grouping vertices corresponding to cosets; this is to make the picture more readable).

To prove that $\Aut(\Gamma) \subseteq B \oplus C$, it remains to show that $O$ is a fixed block of $\Aut(\Gamma)$. This may be achieved by suitable rearrangement of colours of edges in $\Gamma_1$, so that the quadruple of colour degrees of vertices in $O$ is different than that in $Q$ (Note, that because $\Aut(\Gamma)$ is transitive on its orbits, this quadruple is the same for all vertices in the given orbit). Such a rearrangement is impossible only in one case, when both graphs $\Gamma_i$ are isomorphic, all $4$ colours are used, and degree in each colour is the same. It follows that, according to Lemma~\ref{lem:Aalpha}, this happens only in the case when $B$ and $C$ are isomorphic with $Z_3^2$. Moreover, it follows that, in such a case,  
$A$ is a group of the form $Z_{3}^2[H]\oplus_\phi Z_{3}^2[H]$, and since the orbits are adjacent, $H = Z_3$ or $H=I_9$. Yet, by Lemma~\ref{lem:z32z32}, $A\in GR(4)$ in such a case, which completes the proof for the adjacent orbits.

Now, assume that $O$ and $Q$ are not adjacent. Then, by the condition (2), both $B$ and $C$ are elementary abelian $2$-groups of the form $Z_2^m$. Moreover, since $A$ is nontrivial, at least for one of these group $m>0$. In this case, the involutions are trivial, so we do not need any special construction to prevent them. Whence, in this case, all the edges between $X$ and $O$ other than the edges guaranteeing parallel action between the cosets are coloured $0$. Since no group of the form $Z_2^m$, $m>0$, has $4n+1$ elements for any $n$, as in the previous case, the colours of edges in $\Gamma_i$ can be rearranged so that to ensure that $O$ is fixed block of $\Aut(\Gamma)$. Then the result follows as before, completing the proof.
\end{proof}

\emph{Remark 1.} For future reference note that the edges between the orbits are coloured in at most there colours $0,1,2$. This includes the case covered by Lemma~\ref{lem:z32z32}.

\emph{Remark 2.} The assumption that $A$ is nontrivial is only to exclude the exceptional case of the trivial permutation group acting on exactly 2 elements, which (because of lack of room) is not representable by any 2-element graph.

Now we prove the ``if'' part.

\begin{Lemma}\label{lem:ifpart}
If $A$ is a nontrivial abelian permutation group satisfying condition $(1)$ and $(2)$ of the theorem, then $A\in GR(4)$. 
\end{Lemma}

\begin{proof}
The proof is by induction on the number of orbits $r$ of $A$. If $r=1$, $A$ is transitive, and condition (2) means that $A$ is an elementary abelian $2$-group, since in this case $A_i^* = A$. By Lemma~\ref{lem:z2n} $A\in GR(4)$. If $A$ has $2$ orbits, then the results holds by Lemma~\ref{lem:2orb}.

Now, suppose that $A$ has $r>2$ orbits, and the result holds for all groups with the number of orbits less than $r$. 

Consider an arbitrary orbit $X_i$ and the decomposition of $A$ with regard to this orbit, that is, let $A = A_i[A_i^*] \oplus_\phi  B[B']$, where $B$ is the restriction of $A$ to $X\setminus X_i$. Since $A$ is nontrivial, we may assume in addition that $B$ is nontrivial ($A_i$ may happen to be a fixed point). 
Let $\clo{B}$ be the $2$-orbit closure of $B$. By Lemma~\ref{lem:clo},  $\clo{B}$ is abelian, and it has $r-1$ orbits. Moreover, it satisfies condition~(2), providing all isolated orbits in $\clo{B}$ are those isolated in $A$. Let us continue under this additional assumption.

Then, by the induction hypothesis, $\clo{B}\in GR(4)$, and there exists a $4$-coloured graph $\Gamma_2$ on the set of vertices $X\setminus X_i$ representing $\clo{B}$.  We construct a graph $\Gamma$ on $X$ representing $A$. Let $\Gamma_1$ be a $4$-coloured graph on $X_i$ representing $(A_i)^+$ (by Lemma~\ref{lem:Aalpha}). We may assume that both the graphs $\Gamma_1$ and $\Gamma_2$ are connected in colours $2$ and $3$ (meaning every two vertices in $\Gamma_i$ are connected by a path using only edges of colour~2 or 3; this may be achieved by a suitable change of colours). For each orbit $X_j$ of $\clo{B}$ (that are exactly the orbits of $A$ other than $X_i$), we put the edges coloured $1$ between $X_i$ and $X_j$ joining corresponding cosets in parallel manner as in the proof of Lemma~\ref{lem:2orb}. The remaining edges joining the vertices of $X_i$ and $X\setminus X_i$ are coloured $0$. Obviously, $\Aut(\Gamma) \supseteq A$, and since $\clo{B}$ is intransitive, each automorphism of $\Gamma$ preserves the orbit $X_i$. Thus,  
$\Aut(\Gamma) \subseteq (A_i)^+ \oplus \clo{B}$. 

We prove that no nontrivial involution on $X_i$ is admitted, that is, $\Aut(\Gamma)$ $\subseteq A_i \oplus \clo{B})$. Indeed, if
$A_i/A_i^*$ is an elementary abelian $2$-group, then by (2), so is $A_i$, and the involution is trivial. Then $(A_i)^+ = A_i$, and the claim is obvious. So, we may assume that $A_i/A_i^*$ is not an elementary abelian $2$-group. Then by Lemma~\ref{lem:isol}, $X_i$ is not isolated, which means that there is an orbit $X_j$, $j\neq i$, adjacent to $X_i$. In particular, $A_i/A_i^j$ is not an elementary abelian $2$-group. Consider the restriction of $A$ to $X_i\cup X_j$, which can be presented in the form  $A_i[A_i^j] \oplus_\phi A_j[A_j^i]$. Similarly as in 
Lemma~\ref{lem:2orb}, we infer that there is $x\in A_i$ such that $x$ and $x^{-1}$ lie in different cosets of $A_i/A_i^j$. Now, $A_j/A_j^i$ is regular (as it is transitive and abelian, so if an automorphism of $\Gamma$ fixes $A_j^i$, it fixes all the cosets of $A_j/A_j^i$. Because of the edges between $X_i$ and $X_j$ guaranteeing a parallel action on cosets, we infer that if an automorphism of $\Gamma$ fixes $A_i^j$, then it fixes all the cosets in $A_i/A_i^j$. Consequently, there is no automorphism of $\Gamma$ whose restriction to $X_i$ would be the involution. This proves our claim. 

The construction ensures that for any two orbits $X_j$ and $X_k$ of $A$, the  the restriction of $\Aut(\Gamma)$ to $X_j\cup X_k$ is the same as the restriction of $A$ to $X_j\cup X_k$. By Lemma~\ref{lem:2equal}, $\Aut(\Gamma) = A$, as required.    

Thus, we have proved that $A\in GR(4)$, under the conditions $(**)$ \emph{that all isolated orbits in $\clo{B}$ are those isolated in $A$, and that $B$ has a nontrivial orbit}. Consider now the general situation. If there is a trivial orbit (fixed point) in $A$, we may take this orbit as $X_i$ above, and the result follows (because the conditions $(**)$ are satisfied). Otherwise, if there is any isolated orbit in $A$, than we may take it as $X_i$, and again conditions $(**)$ are satisfied, and the result follows. The result also follows in any case when there is an orbit $X_i$ such that $(**)$ are satisfied. All it remains to consider is the situation when $A$ has an even number $r$ of orbits, all nontrivial, and paired in such a way that for every orbit $X_i$ there is a unique orbit $X_j$ in $A$ such that $X_i$ and $X_j$ are adjacent. 

If $r=2$, then the result follows by Lemma~\ref{lem:2orb}. If $r\geq 4$, we take a pair of adjacent orbits $X_i$ and $X_j$, put $Y=X_i\cup X_j$, and $Z=X\setminus Y$, and  decompose  $A$ with regard to $Y$ and $Z$: $A = C[C'] \oplus_\phi  B[B']$. Now, the proof is the same as in the case of decomposing with regard to a chosen orbit $X_i$, with the natural modification for $C$ consisting of two orbits, and using Lemma~\ref{lem:2orb} rather than Lemma~\ref{lem:Aalpha}. This makes the proof simpler, since we may omit the part concerning involutions. 
An additional case is created for $n=4$, since then we need to use a more sophisticated colouring of edges to prevent transposing sets $Y$ and $Z$. In this case both $C$ and $B$ consists each of two (adjacent) orbits, and the problem arises when $\Gamma_1$ and $\Gamma_2$, representing $C$ and $B$, respectively, are isomorphic as coloured graphs. Then we make use of the Remark~1 following the proof of Lemma~\ref{lem:2orb}. According to this remark we may assume that in the graph $\Gamma_1$ representing $C$ the edges between the orbits are in colours $0,1,2$,  while for the graph $\Gamma_2$ representing $B$ the edges between the orbits are in colours $1,2,3$. Then $\Gamma_1$ and $\Gamma_2$ are no longer isomorphic, and the construction works also in this case. This completes the proof of the lemma.
\end{proof}

Now, Theorem~\ref{th:main} follows immediately by Lemmas~\ref{lem:oifpart} and
\ref{lem:ifpart}. In fact, we have proved something more, namely, that in case of abelian permutation groups ``four colours suffices''.

\begin{Corollary}
An abelian group $A$ is the automorphism group of a coloured graph  if and only if it is the automorphism group of a complete graph whose edges are coloured in at most $4$ colours.
\end{Corollary}

We note that, by the remark following Lemma~\ref{lem:z2n}, the bound 4 above is sharp. In fact, using, e.g., the construction of the direct sum with the summand $Z_2^3$ one may obtain an infinite family of abelian permutation groups requiring $4$ colours to be representable by a coloured graph.

\section{Characterization of $2$-closed abelian permutation groups}

We show that the notions ``2-closed'' and ``2-orbit-closed'' coincide. The result is an essential characterization, since checking 2-orbit-closure is computationally more direct and much easier than checking, in general, 2-closure. 

\begin{Theorem}\label{th:main2}
Let $A$ be an abelian permutation group. Then $A$ is the automorphism group of a coloured directed graph if and only if 
$A$ is 2-orbit-closed. 
\end{Theorem}

\begin{proof}
To prove this result we follow the approach in the proof of Theorem~\ref{th:main}. In fact, the proof is simpler because we can start the induction from $r=1$ orbits, and there is no problem of involution and special cases, as in the previous proof. In addition, we may omit the assumption about nontriviality of $A$, since the result holds for trivial groups, as well. (In particular, $I_n\in DGR(2)$ can be represented by the directed graph consisting of one directed path of length $n-1$). As the arguments are similar, we give here only a sketch, referring the reader for details to the 
the previous proof.

The ``only if part'' is immediate by the remark before Lemma~\ref{lem:2equal}. 
We prove the ``if part'' by induction on the number of orbits $r$ of $A$. If $r=1$, $A$ is transitive, and hence regular,
and by the result of Babai \cite{bab1}, every nontrivial regular abelian group $A\in DGR(2)$, except for $A=Z_2^n$ with $n=2,3,4$, and $A= Z_3^2$.
In the first case, by Lemma~\ref{lem:z234}, $Z_2^n \in DGR(4)$ (and again $Z_2^3$ requires $4$ colours; here we have $x=x^{-1}$ and the cases of directed and undirected graphs are the same). In the second case, $Z_3^2\in DGR(4)$ by Lemma~\ref{lem:Aalpha}. In fact, it can be easily seen that $Z_3^2 \in DGR(3)$ (see the right hand side of Figure~\ref{fig2}).


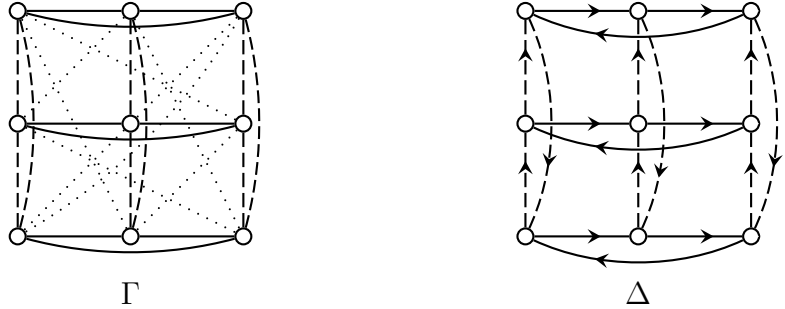
\begin{figure}
  \psset{unit=1.5cm, radius=0.12cm} 

  \begin{center}
\begin{pspicture}(7,3.5)(0,0.3) 


\Cnode(0.2,1){v00}  
\Cnode(0.2,2){v01}  
\Cnode(0.2,3){v02} 

\Cnode(1.2,1){v10}  
\Cnode(1.2,2){v11}  
\Cnode(1.2,3){v12} 

\Cnode(2.2,1){v20}  
\Cnode(2.2,2){v21}  
\Cnode(2.2,3){v22} 

\ncline{v00}{v10}
\ncline{v10}{v20}
\ncarc[arcangle=-15]{v00}{v20}

\ncline{v01}{v11}
\ncline{v11}{v21}
\ncarc[arcangle=-15]{v01}{v21}

\ncline{v02}{v12}
\ncline{v12}{v22}
\ncarc[arcangle=-15]{v02}{v22}

\ncline[linestyle=dashed, dash=5pt 2pt]{v00}{v01}
\ncline[linestyle=dashed, dash=5pt 2pt]{v01}{v02}
\ncarc[linestyle=dashed, dash=5pt 2pt,arcangle=15]{v02}{v00}

\ncline[linestyle=dashed, dash=5pt 2pt]{v10}{v11}
\ncline[linestyle=dashed, dash=5pt 2pt]{v11}{v12}
\ncarc[linestyle=dashed, dash=5pt 2pt,arcangle=15]{v12}{v10}

\ncline[linestyle=dashed, dash=5pt 2pt]{v20}{v21}
\ncline[linestyle=dashed, dash=5pt 2pt]{v21}{v22}
\ncarc[linestyle=dashed, dash=5pt 2pt,arcangle=15]{v22}{v20}

\ncline[linestyle=dotted]{v00}{v11}
\ncline[linestyle=dotted]{v11}{v22}
\ncarc[linestyle=dotted, arcangle=-15]{v00}{v22}

\ncline[linestyle=dotted]{v02}{v21}
\ncline[linestyle=dotted]{v02}{v10}
\ncline[linestyle=dotted]{v10}{v21}

\ncline[linestyle=dotted]{v20}{v12}
\ncline[linestyle=dotted]{v20}{v01}
\ncline[linestyle=dotted]{v01}{v12}

 \psset{arrowscale=1.8,
arrowlength=0.9,
ArrowInsidePos=0.7,ArrowInside=->}

\Cnode(4.7,1){w00}  
\Cnode(4.7,2){w01}  
\Cnode(4.7,3){w02} 

\Cnode(5.7,1){w10}  
\Cnode(5.7,2){w11}  
\Cnode(5.7,3){w12} 

\Cnode(6.7,1){w20}  
\Cnode(6.7,2){w21}  
\Cnode(6.7,3){w22} 

\ncline{w00}{w10}
\ncline{w10}{w20}
\ncarc[arcangle=25]{w20}{w00}

\ncline{w01}{w11}
\ncline{w11}{w21}
\ncarc[arcangle=25]{w21}{w01}

\ncline{w02}{w12}
\ncline{w12}{w22}
\ncarc[arcangle=25]{w22}{w02}

\ncline[linestyle=dashed, dash=5pt 2pt]{w00}{w01}
\ncline[linestyle=dashed, dash=5pt 2pt]{w01}{w02}
\ncarc[linestyle=dashed, dash=5pt 2pt,arcangle=25]{w02}{w00}

\ncline[linestyle=dashed, dash=5pt 2pt]{w10}{w11}
\ncline[linestyle=dashed, dash=5pt 2pt]{w11}{w12}
\ncarc[linestyle=dashed, dash=5pt 2pt,arcangle=25,ArrowInsidePos=0.75]{w12}{w10}

\ncline[linestyle=dashed, dash=5pt 2pt]{w20}{w21}
\ncline[linestyle=dashed, dash=5pt 2pt]{w21}{w22}
\ncarc[linestyle=dashed, dash=5pt 2pt,arcangle=25]{w22}{w20}

\rput(1.2,0.5){{\large $\Gamma$}}
\rput(5.7,0.5){{\large $\Delta$}}


\end{pspicture}
\end{center} \caption{Simple and directed graphs:  $\Aut(\Gamma) = (Z_3^2)^+$  and $\Aut(\Delta) = Z_3^2$. }\label{fig2}
\end{figure}

Now, suppose that $A$ has $r>1$ orbits, and the result holds for all groups with the number of orbits less than $r$. 

Consider an arbitrary orbit $X_i$ and the decomposition of $A$ with regard to this orbit, that is, let $A = A_i[A_i^*] \oplus_\phi  B[B']$, where $B$ is the restriction of $A$ to $X\setminus X_i$. 
Let $\clo{B}$ be the $2$-orbit closure of $B$. By Lemma~\ref{lem:clo},  $\clo{B}$ is abelian, and it has $r-1$ orbits. 

It follows, by the induction hypothesis, that $\clo{B}\in DGR(4)$, and there exists a $4$-coloured digraph $\Gamma_2$ on the set of vertices $X\setminus X_i$ representing $\clo{B}$.  We construct a digraph $\Gamma$ on $X$ representing $A$. Let $\Gamma_1$ be a $4$-coloured digraph on $X_i$ representing $A_i$ (which exists by the proof for the case $r=1$). We may assume that both the graphs $\Gamma_1$ and $\Gamma_2$ are connected (as undirected graphs) in colours $2$ and $3$ (this may be achieved by suitable change of colours). For each orbit $X_j$ of $\clo{B}$, we put the edges coloured $1$ from $X_i$ to $X_j$ joining corresponding cosets in parallel manner as in the proof of Lemma~\ref{lem:2orb}. We assume that these edges are directed from $X_i$ to $X_j$.
The remaining edges joining the vertices of $X_i$ and $X\setminus X_i$ are coloured $0$. 

Then, obviously, $\Aut(\Gamma) \supseteq A$, Moreover, since $\Gamma_1$ and $\Gamma_2$ are connected in colours $2$ and $3$, and edges in colour $1$ between $\Gamma_1$ and $\Gamma_2$ are directed from $\Gamma_1$ to $\Gamma_2$,  each automorphism of $\Gamma$ preserves the orbit $X_i$. Thus,  
$\Aut(\Gamma) \subseteq A_i \oplus \clo{B}$. 

The construction ensures that for any two orbits $X_j$ and $X_k$ of $A$, the  the restriction of $\Aut(\Gamma)$ to $X_j\cup X_k$ is the same as the restriction of $A$ to $X_j\cup X_k$. Hence, by Lemma~\ref{lem:2equal}, $\Aut(\Gamma) = A$, as required. 
This completes the proof.
\end{proof}

Again, what we have proved in addition is that  ``four colours suffices'', and that the bound $4$ below is sharp.

\begin{Corollary}
An abelian group $A$ is the automorphism group of a coloured directed graph  if and only if it is the automorphism group of a complete directed graph $($without loops$)$ whose edges are coloured in at most $4$ colours.
\end{Corollary}


\end{document}